\documentclass[12pt]{article}
\RequirePackage{ifpdf}
\RequirePackage{graphicx}

\newtheorem{theorem}{Theorem}
\newtheorem{example}{Example}

\newtheorem{lemma}{Lemma}

\newtheorem{remark}{Remark}

\newtheorem{question}{Question}
\newtheorem{notation}{Notation}
\newenvironment{proof}{{\bf Proof.}}{\hspace*{1mm}\hfill\rule{2mm}{2mm}}

\newtheorem{pretheorema}{{\bf Theorem}}

\newtheorem{prelemmab}{{\bf Lemma}}

\newtheorem{prepropositionc}{{\bf Proposition}}

\input amssym.tex
 \title{Silver block intersection graphs of Steiner $2$-designs}
 \author{{\sc A. Ahadi}\thanks{Department of Mathematical Sciences, Sharif University of Technology,
P. O. Box 11155-9415, Tehran, I. R. Iran  arash\_ahadi5@yahoo.com},
{\sc Nazli  Besharati}\thanks{Department of Mathematical Sciences,  Payame Noor University, P.O. Box 19395-3697,
Tehran, I. R. Iran n\_besharati@yahoo.com},\\
{\sc E.S. Mahmoodian}\thanks{Department of Mathematical Sciences, Sharif University of Technology,
P. O. Box 11155-9415, Tehran, I. R. Iran emahmood@sharif.edu, Corresponding author},
{\sc M. Mortezaeefar}\thanks{Department of Mathematical Sciences, Sharif University of Technology,
P. O. Box 11155-9415, Tehran, I. R. Iran m.mortezaeefar@gmail.com}}

\date{}

\begin{document}
\maketitle

%
\begin{abstract}
For a block design $\cal{D}$, a series of {\sf block intersection
graphs} $G_i$, or $i$-{\rm BIG}($\cal{D}$), $i=0, \dots, k$ is
defined in which the vertices are the blocks of $\cal{D}$, with
two vertices adjacent if and only if the corresponding blocks
intersect in exactly $i$ elements.  A silver graph $G$ is defined
with respect to a maximum independent set of $G$, called an {\sf
$\alpha$-set}. Let $G$ be an $r$-regular graph and $c$ be a
proper $(r + 1)$-coloring of  $G$. A vertex $x$ in $G$ is said to
be {\sf rainbow} with respect to $c$ if every color appears in
the closed neighborhood $N[x] = N(x) \cup \{x\}$. Given an
$\alpha$-set $I$ of $G$, a coloring $c$ is said to be silver with
respect to $I$ if every $x\in I$ is rainbow with respect to $c$.
We say $G$ is {\sf silver} if it admits a silver coloring with
respect to some $I$. Finding silver graphs is of interest,
for a motivation and progress in silver
graphs see~\cite{MR2452764}~and~\cite{MR1740669}.
We investigate conditions for $0$-{\rm BIG}($\cal{D}$) and
$1$-{\rm BIG}($\cal{D}$) of Steiner  $2$-designs
${\cal{D}}=S(2,k,v)$ to be silver.

{\bf keywords: }{Silver coloring,   Block intersection graph, 
Steiner  $2$-design, and Steiner triple system}\\

{\textbf{Subject class:} }{ 05C15, 05B05, 05B07,  and 05C69}
\end{abstract}
\section{Introduction and preliminaries}
We follow standard notations and concepts from design theory. For
these, one may refer to, for example,~\cite{MR2246267}
and~\cite{MR1871828}.

A {\sf $2$-$(v,k,\lambda)$ design} $(2 < k < v)$
   is a pair $(V,\cal{B})$ where $V$ is a
$v$-set and $\cal{B}$ is a collection of $b$ \  $k$-subsets of
$V$ (blocks) such that any $2$-subset of $V$ is contained in
exactly $\lambda$ blocks.
A  $2$-$(v,k,1)$ design is called {\sf Steiner $2$-design} and
is denoted by $S(2,k,v)$.
An $S(2, 3, v)$ is a {\sf Steiner triple system} or  STS$(v)$.
A design with $b = v$  is a {\sf
symmetric $(v, k, \lambda)$-design}.
A symmetric $S(2, k, v)$ is called a {\sf projective plane}. If
$k$ is the size of the blocks then $n:=k-1$ is called the order
of the plane. This design is usually denoted by {\rm PG}$(2,n)$.
 A $2$-$(n^2,n,1)$ design is called an {\sf affine plane}.
For such design we use the notation {\rm AG}$(2,n)$.

A {\sf partial parallel class} is a set of blocks that contains
no element of the design more than once. A {\sf parallel class}
(PC) or a {\sf resolution class} in a design is a set of blocks
that partition the set of elements $V.$ A {\sf near parallel
class} is a partial parallel class missing a single element. A
{\sf resolvable balanced incomplete block design} is a $2$-$(v,
k,\lambda)$ design whose blocks can be partitioned into parallel
classes. The notation RBIBD$(v,k,\lambda)$ is commonly used.
An affine plane of order $n$ is an  RBIBD$(n^2,n,1)$.
A resolvable {\rm STS}$(v)$ together with a resolution of its blocks is
called a {\sf Kirkman triple system}, {\rm KTS}$(v)$.
%

Given a design $\cal{D}$, a series of {\sf block intersection
graphs} $G_i$, or $i$-{\rm BIG}, $i=0, \dots, k$ can be defined
in which the vertices are the blocks of $\cal{D}$, with two
vertices are adjacent if and only if the corresponding blocks
intersect in exactly $i$ elements.
\begin{example}
\label{STS7STS9} For {\rm STS}$(7)$, $0$-{\rm BIG} is empty graph
and $1$-{\rm BIG} is $K_7$. For {\rm STS}$(9)$, $0$-{\rm BIG} is
disconnected and consists of four disjoint $K_3$'s and $1$-{\rm
BIG} is $K_{3,3,3,3}.$
\end{example}
The study of $i$-{\rm BIG}($\cal{D}$) is useful in characterizing
block designs.  Some researchers have studied properties of
various kinds of block intersection graphs, see for example
\cite{MR1140783}, \cite{MR1079733}, \cite{MR1170773},
\cite{MR1312454}, \cite{MR1198541},
\cite{MR1691408}, \cite{MR2071911}, and \cite{MR1675136}.

A graph  of order $v$ is { \sf strongly regular}, denoted by
${\rm SRG}(v, k, \lambda, \mu)$, whenever it is not complete or
edgeless and, (i) each vertex is adjacent to $k$ vertices, (ii)
for each pair of adjacent vertices there are $\lambda$  vertices
adjacent to both, (iii) for each pair of non-adjacent vertices
there are $\mu$  vertices adjacent to both.
\begin{remark}
\label{srg1} Let $G_i$ be the $i$-block intersection graph of an
$S(2,k,v)$. Then  for each $i=2,3, \dots, k$, the graph $G_i$  is
empty. So we consider only $G_0$ and $G_1$. Graphs $G_0$ and $G_1$
are complements of each other. $G_1$ is an ${\rm
SRG}(b,k(r-1),r-2+(k-1)^2,k^2)$ and $G_0$ is an ${\rm
SRG}(b,b-k(r-1)-1,b-2k(r-1)+ k^2-2, b-2kr+k^2+r-1)$  (see Chapter~21 of~\cite{MR1871828}).
\end{remark}
%
In a graph $G=(V,E)$ an {\sf independent set}  is a subset of
vertices  no two of which  are adjacent. The {\sf independence
number} $\alpha (G)$ is the cardinality of a largest set of
independent vertices. We refer to any maximum independent set of
a graph as an $\alpha$-{\sf set}. Let $c$ be a proper $(r +
1)$-coloring of an $r$-regular graph $G$. A vertex $x$ in $G$ is
said to be {\sf rainbow} with respect to $c$ if every color
appears in the closed neighborhood $N[x] = N(x)\cup \{x\}$.
%
Given an $\alpha$-set $I$ of $G$ the coloring $c$ is said to be
{\sf silver} with respect to $I$ if every $x\in I$ is rainbow with
respect to $c$. We say $G$ is silver if it admits a silver
coloring with respect to some $\alpha$-set. If all vertices of
$G$ are rainbow, then $c$ is called a {\sf totally silver}
coloring of $G$ and $G$ is said to be totally silver. Note that
the definition of silver coloring depends on the chosen
$\alpha$-set. For example in Figure~\ref{diagonal}, a graph $G$ is
shown which is silver when the $\alpha$-set (the bold vertices)
is taken as in the left, but it does not have any silver coloring
with the $\alpha$-set taken as on the right hand side.
\vspace*{-1mm}
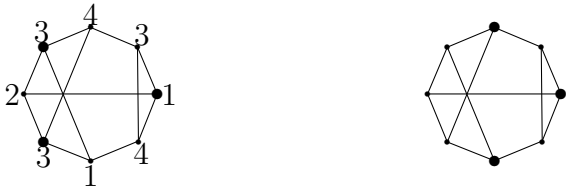
\begin{figure}[ht]
\def\emline#1#2#3#4#5#6{%
\put(#1,#2){\special{em:moveto}}%
\put(#4,#5){\special{em:lineto}}}
\def\newpic#1{}
%
%
%
\unitlength 0.12mm
\special{em:linewidth 0.4pt}
\linethickness{0.4pt}
%
\begin{picture}(150,150)(-160,0)
%
\put(207,73){\circle*{12}} \put(185,125){\circle*{5}}
\put(133,147){\circle*{5}} \put(81,125){\circle*{12}}
\put(59,73){\circle*{5}} \put(81,20){\circle*{12}}
\put(133,-1){\circle*{5}} \put(186,20){\circle*{5}}
\emline{207}{73}{1}{185}{125}{2}
\emline{207}{73}{1}{59}{73}{2}
\emline{207}{73}{1}{186}{21}{2}
\emline{185}{125}{1}{133}{147}{2}
\emline{185}{125}{1}{186}{21}{2}
\emline{133}{147}{1}{81}{125}{2}
\emline{133}{147}{1}{81}{20}{2}
\emline{81}{125}{1}{59}{73}{2}
\emline{81}{125}{1}{133}{-1}{2}
\emline{59}{73}{1}{81}{20}{2}
\emline{81}{20}{1}{133}{-1}{2}
\emline{133}{-1}{1}{186}{21}{2}
\put(189,7){\makebox(0,0)[cc]{$4$}}
\put(133,-17){\makebox(0,0)[cc]{$1$}}
\put(81,4){\makebox(0,0)[cc]{$3$}}
\put(46,72){\makebox(0,0)[cc]{$2$}}
\put(79,142){\makebox(0,0)[cc]{$3$}}
\put(190,139){\makebox(0,0)[cc]{$3$}}
\put(220,72){\makebox(0,0)[cc]{$1$}}
\put(133,161){\makebox(0,0)[cc]{$4$}}
\end{picture}
%
\def\emline#1#2#3#4#5#6{%
\put(#1,#2){\special{em:moveto}}%
\put(#4,#5){\special{em:lineto}}}
\def\newpic#1{}
\unitlength 0.12mm
\special{em:linewidth 0.4pt}
\linethickness{0.4pt}
\begin{picture}(150,150)(-400,0)
%
\put(207,73){\circle*{12}} \put(185,125){\circle*{5}}
\put(133,147){\circle*{12}} \put(81,125){\circle*{5}}
\put(59,73){\circle*{5}} \put(81,20){\circle*{5}}
\put(133,-1){\circle*{12}} \put(186,20){\circle*{5}}
\emline{207}{73}{1}{185}{125}{2}
\emline{207}{73}{1}{59}{73}{2}
\emline{207}{73}{1}{186}{21}{2}
\emline{185}{125}{1}{133}{147}{2}
\emline{185}{125}{1}{186}{21}{2}
\emline{133}{147}{1}{81}{125}{2}
\emline{133}{147}{1}{81}{20}{2}
\emline{81}{125}{1}{59}{73}{2}
\emline{81}{125}{1}{133}{-1}{2}
\emline{59}{73}{1}{81}{20}{2}
\emline{81}{20}{1}{133}{-1}{2}
\emline{133}{-1}{1}{186}{21}{2}
\end{picture}
%
\caption{A silver coloring of a graph} \label{diagonal}
\end{figure}
%

There are many different version of rainbow colorings in the
literature, for example see \cite{AminiEsperetHeuvel}, \cite{MR1076016},
 \cite{KantLeeuwen}, and  \cite{MR2378044}.
For a
motivation and progress in silver graphs
see~\cite{MR2452764}~and~\cite{MR1740669}.
In fact silver graphs are closely related to a concept in graph
coloring, called defining set. Let $c$ be a proper $k$-coloring
of a graph $G$ and let $S\subseteq V (G)$. If $c$ is the only
extension of $c|_S$ to a proper $k$-coloring of $G$, then $S$ is called
a {\sf defining set of $c$}. The minimum size of a defining set
among all $k$-colorings of $G$ is called a {\sf defining number}
and denoted by {\rm def}$(G,k)$. A more general survey of
defining sets in combinatorics appears in \cite{MR2011736}. Let
$G$ be an $r$-regular graph, then $G$ is silver if and only if
{\rm def}$(G,r+1) =| V(G)|-\alpha(G)$. In~\cite{MR1740669} an
open problem is raised:

\begin{question}
\label{silver question}
Find classes of $r$-regular graphs $G$, for which {\rm
def}$(G,r+1) = | V(G)|-\alpha(G)$, i.e. determine
 classes of all silver graphs.
\end{question}

A silver cube is a silver graph $G=K_n^d$,
  the Cartesian power of the complete graph $K_n$. Silver cubes
  are
  generalizations of silver matrices, which are
 $n \times n$ matrices where each symbol in $\{1,2,\dots,2n-1\}$ appears in
 either the $i$-th row or the $i$-th column of the matrix. In~\cite{MR2452764}
 some algebraic constructions and a product construction of silver cubes are given.
 They  show the relation of these cubes to codes over
 finite fields,
 dominating sets of a graph, Latin squares, and finite geometry.
 In particular the Hamming codes are used to produce a totally silver cube
 and the bound for the best binary codes is used to prove the non-existence of
 silver cubes for a large class of parameters with $n=2.$

To study Question~\ref{silver question}, here we consider  $i$-{\rm BIGs} of designs. First we give some
examples of designs with silver $i$-{\rm BIGs}.
\begin{example}
\label{symmetric-AG}
In any symmetric $(v, k, \lambda)$-design ${\cal D}$, every two
distinct blocks have exactly $\lambda$ elements in common, so for
$ 0 \leq i \leq k$, $i \neq \lambda$, $i$-{\rm BIG}$(\cal{D})$ is
empty graph, and $\lambda$-{\rm BIG}$(\cal{D})$ is complete
graph. Hence all of these graphs are totally silver. Specifically
for each $k$ and $0 \leq i \leq k+1$, $i$-{\rm BIG}$(S(2,k+1,
k^2+k+1))$ is totally silver.

If $\cal{D}$ is an ${\rm AG}(2,n)$, then $G_0 = 0$-{\rm
BIG}$(\cal{D})$ consists of $(n+1)$ disjoint $K_n$'s, so it is
totally silver, and $G_1 = 1$-{\rm
BIG}${(\cal{D})}=K_{\underbrace{n,n,\ldots,n}_{n+1}}$,  is silver.
\end{example}
In this paper we prove the following results: If an $S(2,k,v)$
contains a parallel class, then  a necessary condition for
$1$-{\rm BIG({$S(2,k,v)$}}) to be  silver is $ k^2\mid v.$ For
each admissible $v=9m$ we construct a ${\cal{D}}_1$= {\rm
KTS}$(v)$,  such that $1$-{\rm BIG}$({\cal{D}}_1)$ is silver. And
in general for each $k$ and $v$ where an  ${\rm AG}(2,k)$ and an
{\rm RBIBD}$(v,k,1)$ exist we construct a ${\cal{D}}^*= {\rm
RBIBD}(kv,k,1)$  such that $1$-{\rm BIG}$({\cal{D}}^*)$ is silver.
Also a lower bound for $\alpha(G_1)$ is given in order for a
$1$-{\rm BIG({$S(2,k,v)$}}) to be  silver.
For any admissible $v$, the existence of a silver $1$-{\rm
BIG($S(2,k,v)$)} which possesses a  maximum possible independent
set, i.e. of size $\frac{v}{k}$ or $\frac{v-1}{k}$, is settled.
We prove that for $ v > k^3-2k^2+2k$ there is no silver $0$-{\rm
BIG({$S(2,k,v)$}}). Also we settle the question of existence of
silver $0$-{\rm BIG(STS$(v)$)} for all admissible $v$.

Since every vertex of $i$-{\rm BIG}$(\cal{D})$ corresponds to
a block of $\cal{D}$, we will mostly refer to them as ``blocks''
rather than vertices. The following notation will be used in our
discussion.
 Let $G$ be a graph and $I$ be an $\alpha$-set
of $G$. For each $ i= 1,\dots,|I|$, we let
\vspace*{-1.6mm}
$$X_i := \{u| u \in V(G) \setminus I, \ u \ {\rm is
\ adjacent \ to \ exactly}  \ i \  {\rm vertices  \ of}  \ I\}.$$

%

\section{One block intersection graphs}

The following is a necessary condition for $1$-{\rm
BIG}$({\cal{D}})$ of a Steiner system  ${\cal{D}}=S(2,k,v)$ with
$\alpha(G_1)= \frac{v}{k}$, to be silver.

\begin{theorem}
\label{forallk} Let $\cal{D}$ be an $S(2,k,v)$, which has a
parallel class, and let $G_1$ be $1$-${\rm BIG(\cal{D}})$. A
necessary condition for  $G_1$ to be  silver  is  $ k^2\mid v.$
\end{theorem}
\begin{proof}
 $G_1$ is a $\frac{k(v-k)}{(k-1)}$--regular graph.
Let $I$ be an $\alpha$-set, and assume that  $G_1$ has
a silver coloring  with respect to $I$ with $C$ as  the set of
colors. We have $|I|=\frac{v}{k}$, and  $|C|=
\frac{k(v-k)}{k-1}+1$. Since $|C| > |I|$, a color like $\iota$
exists that is not used in  $I$.
The vertices of $I$ are rainbow, and each vertex with color
$\iota$ from  $V(G_1) \setminus I$,  must be adjacent to $k$
distinct vertices of $I$. Therefore $|I|$ must be a multiple of
$k$, which implies $k^2\mid v$. \end{proof}
\begin{example}
\label{70sts} There are $80$ nonisomorphic ${\rm STS}(15)$s, where
$70$ of them have parallel class~{\rm(see~\cite{MR2246267}, page
32)}. So by Theorem~\ref{forallk}, none of those $70$ has silver
$G_1$.
\end{example}
By Theorem~\ref{forallk}, if $v$ is not a multiple of $9$, then no
silver $1$-${\rm BIG(KTS}(v))$  exists. In the next lemma we show
that for the case $9\mid v$, when a ${\rm KTS}(v)$ exists,  i.e.
$v=18q+9$, there exists a silver $1$-${\rm BIG(KTS}(v))$. This lemma is an
illustration of a general structure  which will be discussed in
Theorem~\ref{RBIBDfork}.

\begin{lemma}
\label{9midvKTS} If  $v \equiv 3 \pmod {6}$, then a ${\cal
K}={\rm KTS}(3v)$ exists such that $1$-${\rm BIG}({\cal K})$ is
silver.
\end{lemma}
\begin{proof}
Let ${\cal A}= {\rm AG}(2,3)={\rm STS}(9)$ with $V({\cal{A}})= \{
(i,j) \ | \ 1 \leq i,j \leq 3 \}$, and denote its parallel
classes by:
$$\begin{array}{ccc}
 \Theta_{0}    \\
\{(1,1),(2,1),(3,1)\}\\
\{(1,2),(2,2),(3,2)\}\\
\{(1,3),(2,3),(3,3)\}\\
 \end{array}$$
$$\begin{array}{ccccc}
 \begin{array}{c c c}
  \Theta_{1}  \\
a_1= \{(1,1),(1,2),(1,3)\}\\
a_2=\{(2,1),(2,2),(2,3)\}\\
a_3=\{(3,1),(3,2),(3,3)\}\\
\end{array}
& &
\begin{array}{ccc}
\Theta_{2} \\
a_4=\{(1,1),(2,2),(3,3)\}\\
a_5=\{(1,3),(2,1),(3,2)\}\\
a_6=\{(1,2),(2,3),(3,1)\} \\
\end{array}
& &
\begin{array}{ccc}
 \Theta_{3}  \\
a_7=\{(1,1),(2,3),(3,2)\}\\
a_8=\{(1,2),(2,1),(3,3)\}\\
a_9=\{(1,3),(2,2),(3,1)\}\\
\end{array}
\end{array}$$
\\

\noindent
Consider a ${\rm KTS}(v)$ ${\cal D} = (V, {\cal B})$, $V =
\{x_{1},x_{2},\ldots,x_{v}\}$  with parallel classes $\pi_1,\pi_2,
\dots, \pi_{\frac{v-1}{2}}$. Using its blocks we construct ${\cal
K}  = (V^{*}, {\cal B^{*}})$,  a  ${\rm KTS}(3v)$ in the
following manner.

The set of elements of ${\cal K}$ is $  V^{*} = \{1,2,3\} \times
V$, and the blocks are introduced in the following $4$ types of
parallel classes, $\Omega_{0,\beta}$, $\Omega_{1,\beta}$,
$\Omega_{2,\beta}$ and $\Omega_{3,\beta}$.
\begin{itemize}
\item $\Omega_{0,\beta}:$ \
 $\Big\{\{(1,x_{i}),(2,x_{i}),(3,x_{i})\}| \ 1\leq i \leq v\Big  \}$.
\end{itemize}
We denote every block of ${\cal {D}}$ by $\{x_i,x_j,x_k\}$, where
$i <j<k$. In the following a label $(m,\beta)$ for each block is
its color, the block with label $(m,\beta)$ is obtained by using
the block $ a_m$ of~${\cal A}$.
\begin{itemize}
\item $\Omega_{1,\beta}$: \
         $\Big\{\{(1,x_i),(1,x_j),(1,x_k)\}_{(1,\beta)}, \
         \{(2,x_i),(2,x_j),(2,x_k)\}_{(2,\beta)},\
          \{(3,x_i),(3,x_j), \\  \hspace*{19mm} (3,x_k)\}_{(3,\beta)}\mid
         \{x_i,x_j,x_k\} \in \pi_{\beta} \Big\}$,
         \ for $1\le \beta \le \frac{v-1}{2}$,
\item  $\Omega_{2,\beta}$: \
        $\Big \{\{(1,x_{i}),(2,x_{j}),(3,x_{k})\}_{(4,\beta)},
         \{(1,x_{k}),(2,x_{i}),(3,x_{j})\}_{(5,\beta)},
         \{(1,x_{j}),(2,x_{k}), \\  \hspace*{19mm} (3,x_{i})\}_{(6,\beta)}\mid
         \{x_i,x_j,x_k\} \in \pi_{\beta}$\Big\}, \ for $1 \le \beta \le \frac{v-1}{2}$,
\item $\Omega_{3,\beta}$: \
         $\Big\{\{(1,x_{i}),(2,x_{k}),(3,x_{j})\}_{(7,\beta)},
         \{(1,x_{j}),(2,x_{i}),(3,x_{k})\}_{(8,\beta)},
         \{(1,x_{k}),(2,x_{j}), \\  \hspace*{19mm} (3,x_{i})\}_{(9,\beta)}\mid
          \{x_i,x_j,x_k\} \in \pi_{\beta}$\Big\}, \ for $1 \le \beta \le \frac{v-1}{2}$.
\end{itemize}
%
%
Figures~\ref{type2}~and~\ref{type4}  demonstrate  the
$4$ types of blocks.
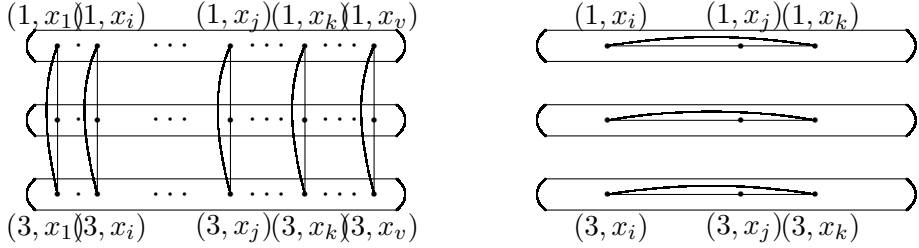
\begin{figure}[ht]
\def\emline#1#2#3#4#5#6{%
\put(#1,#2){\special{em:moveto}}%
\put(#4,#5){\special{em:lineto}}}
\def\newpic#1{}
%
%
%
\unitlength 0.23mm
\special{em:linewidth 0.4pt}
\linethickness{0.4pt}
\begin{picture}(150,150)(-75,0)
%
\put(-17,78){\circle*{3}}
\put(-17,35){\circle*{3}}
\put(-17,121){\circle*{3}}
\put(60,121){\circle*{3}}
\put(60,78){\circle*{3}}
\put(60,35){\circle*{3}}
\put(103,121){\circle*{3}}
\put(103,78){\circle*{3}}
\put(103,35){\circle*{3}}
\put(-40,121){\circle*{3}}
\put(-40,78){\circle*{3}}
\put(-40,35){\circle*{3}}
\put(143,121){\circle*{3}}
\put(143,78){\circle*{3}}
\put(143,35){\circle*{3}}
\put(-53,130){\circle*{0}}
\put(-53,87){\circle*{0}}
\put(-53,69){\circle*{0}}
\put(-53,112){\circle*{0}}
\put(-53,44){\circle*{0}}
\put(-53,26){\circle*{0}}
\put(156,44){\circle*{0}}
\put(156,26){\circle*{0}}
\put(156,130){\circle*{0}}
\put(156,112){\circle*{0}}
\put(156,69){\circle*{0}}
\put(156,87){\circle*{0}}

\emline{-17}{35}{1}{-17}{78}{2}
\emline{-17}{78}{1}{-17}{121}{2}
\bezier{500}(-17, 35)(-32,90)(-17, 121)
\emline{60}{35}{1}{60}{78}{2}
\emline{60}{121}{1}{60}{78}{2}
\bezier{500}(60, 35)(45, 90)(60, 121)
\emline{103}{35}{1}{103}{78}{2}
\emline{103}{121}{1}{103}{78}{2}
\bezier{500}(103,35)(88, 90)(103, 121)
\emline{-40}{121}{1}{-40}{78}{2}
\emline{-40}{35}{1}{-40}{78}{2}
\bezier{500}(-40,35)(-53, 90)(-40, 121)
\emline{143}{121}{1}{143}{78}{2}
\emline{143}{35}{1}{143}{78}{2}
\bezier{500}(143,35)(128, 90)(143, 121)
\emline{156}{130}{1}{-53}{130}{2}
\bezier{500}(156, 130)(166, 121)(156, 112)
\bezier{500}(-53, 130)(-63, 121)(-53, 112)
\bezier{500}(-53, 87)(-63, 78)(-53, 69)
\emline{-53}{87}{1}{156}{87}{2}
\emline{-53}{69}{1}{156}{69}{2}
\bezier{500}(156, 87)(166, 78)(156, 69)
\bezier{500}(-53, 44)(-63, 35)(-53, 26)
\emline{-53}{44}{1}{156}{44}{2}
\emline{-53}{26}{1}{156}{26}{2}
\bezier{500}(156, 44)(166, 35)(156, 26)
\emline{156}{112}{1}{-53}{112}{2}
\put(-10,16){\makebox(0,0)[cc]{\small{$(3,x_i)$}}}
\put(-10,138){\makebox(0,0)[cc]{\small{$(1,x_i)$}}}
\put(62,16){\makebox(0,0)[cc]{\small{$(3,x_j)$}}}
\put(62,138){\makebox(0,0)[cc]{\small{$(1,x_j)$}}}
\put(106,138){\makebox(0,0)[cc]{\small{$(1,x_k)$}}}
\put(106,16){\makebox(0,0)[cc]{\small{$(3,x_k)$}}}
\put(-47,16){\makebox(0,0)[cc]{\small{$(3,x_1)$}}}
\put(146,16){\makebox(0,0)[cc]{\small{$(3,x_v)$}}}
\put(-47,138){\makebox(0,0)[cc]{\small{$(1,x_1)$}}}
\put(146,138){\makebox(0,0)[cc]{\small{$(1,x_v)$}}}
\put(25,78){\makebox(0,0)[cc]{$\cdots$}}
\put(25,35){\makebox(0,0)[cc]{$\cdots$}}
\put(25,121){\makebox(0,0)[cc]{$\cdots$}}
\put(-28,78){\makebox(0,0)[cc]{$\cdot$}}
\put(-28,35){\makebox(0,0)[cc]{$\cdot$}}
\put(-28,121){\makebox(0,0)[cc]{$\cdot$}}
\put(81,78){\makebox(0,0)[cc]{$\cdots$}}
\put(81,35){\makebox(0,0)[cc]{$\cdots$}}
\put(81,121){\makebox(0,0)[cc]{$\cdots$}}
\put(123,78){\makebox(0,0)[cc]{$\cdots$}}
\put(123,35){\makebox(0,0)[cc]{$\cdots$}}
\put(123,121){\makebox(0,0)[cc]{$\cdots$}}
\end{picture}
%
%
\def\emline#1#2#3#4#5#6{%
\put(#1,#2){\special{em:moveto}}%
\put(#4,#5){\special{em:lineto}}}
\def\newpic#1{}
%
%
%
\unitlength 0.23mm
\special{em:linewidth 0.4pt}
\linethickness{0.4pt}
\begin{picture}(150,150)(-190,0)
%
\put(-17,78){\circle*{3}} \put(-17,35){\circle*{3}}
\put(-17,121){\circle*{3}}
\put(60,121){\circle*{3}} \put(60,78){\circle*{3}}
\put(60,35){\circle*{3}}
\put(103,121){\circle*{3}} \put(103,78){\circle*{3}}
\put(103,35){\circle*{3}}
\put(-53,130){\circle*{0}} \put(-53,87){\circle*{0}}
\put(-53,69){\circle*{0}} \put(-53,112){\circle*{0}}
\put(-53,44){\circle*{0}} \put(-53,26){\circle*{0}}
\put(156,44){\circle*{0}} \put(156,26){\circle*{0}}
\put(156,130){\circle*{0}} \put(156,112){\circle*{0}}
\put(156,69){\circle*{0}} \put(156,87){\circle*{0}}

\emline{-17}{78}{1}{60}{78}{2}
\bezier{500}(-17, 78)(43, 88)(103, 78)
\emline{-17}{35}{1}{60}{35}{2}
\bezier{500}(-17, 35)(43, 45)(103, 35)
\emline{60}{121}{1}{103}{121}{2}
\emline{60}{121}{1}{-17}{121}{2}
\emline{60}{78}{1}{103}{78}{2}
\emline{60}{35}{1}{103}{35}{2}
\bezier{500}(103, 121)(43, 131)(-17, 121)
\emline{156}{130}{1}{-53}{130}{2}
\bezier{500}(156, 130)(166, 121)(156, 112)
\bezier{500}(-53, 130)(-63, 121)(-53, 112)
\bezier{500}(-53, 87)(-63, 78)(-53, 69)
\emline{-53}{87}{1}{156}{87}{2}
\emline{-53}{69}{1}{156}{69}{2}
\bezier{500}(156, 69)(166, 78)(156, 87)
\bezier{500}(-53, 44)(-63, 35)(-53, 26)
\emline{-53}{44}{1}{156}{44}{2}
\emline{-53}{26}{1}{156}{26}{2}
\bezier{500}(156, 44)(166, 35)(156, 26)
\emline{156}{112}{1}{-53}{112}{2}
\put(-15,16){\makebox(0,0)[cc]{\small{$(3,x_i)$}}}
\put(-15,138){\makebox(0,0)[cc]{\small{$(1,x_i)$}}}
\put(62,16){\makebox(0,0)[cc]{\small{$(3,x_j)$}}}
\put(62,138){\makebox(0,0)[cc]{\small{$(1,x_j)$}}}
\put(106,138){\makebox(0,0)[cc]{\small{$(1,x_k)$}}}
\put(106,16){\makebox(0,0)[cc]{\small{$(3,x_k)$}}}
\end{picture}
\vspace{-1mm} \caption{Blocks of $\Omega_{0,\beta}$ and
$\Omega_{1,\beta}$} \label{type2}
\end{figure}

%
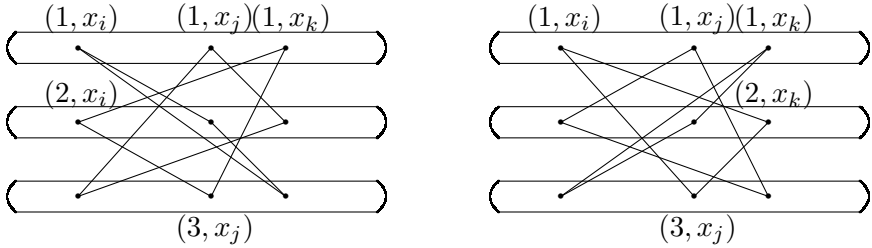
\begin{figure}[ht]
\def\emline#1#2#3#4#5#6{%
\put(#1,#2){\special{em:moveto}}%
\put(#4,#5){\special{em:lineto}}}
\def\newpic#1{}
\unitlength 0.23mm
\special{em:linewidth 0.4pt}
\linethickness{0.4pt}
\begin{picture}(150,150)(-362,6)
%
\put(-17,78){\circle*{3}} \put(-17,35){\circle*{3}}
\put(-17,121){\circle*{3}}
\put(60,121){\circle*{3}} \put(60,78){\circle*{3}}
\put(60,35){\circle*{3}}
\put(103,121){\circle*{3}}
\put(103,78){\circle*{3}}
\put(103,35){\circle*{3}}
\put(-53,130){\circle*{0}}
\put(-53,87){\circle*{0}}
\put(-53,69){\circle*{0}}
\put(156,130){\circle*{0}}
\put(156,69){\circle*{0}}
\put(156,87){\circle*{0}}
\put(-53,44){\circle*{0}}
\put(-53,26){\circle*{0}}
\put(-53,112){\circle*{0}}
\put(156,44){\circle*{0}}
\put(156,26){\circle*{0}}
\put(156,112){\circle*{0}}
\emline{-17}{78}{1}{60}{121}{2}
\emline{-17}{78}{1}{103}{35}{2}
\emline{-17}{35}{1}{60}{78}{2}
\emline{-17}{35}{1}{103}{121}{2}
\emline{60}{121}{1}{103}{35}{2}
\emline{60}{78}{1}{103}{121}{2}
\emline{60}{35}{1}{103}{78}{2}
\emline{60}{35}{1}{-17}{121}{2}
\emline{103}{78}{1}{-17}{121}{2}
\emline{156}{130}{1}{-53}{130}{2}
\bezier{500}(156, 130)(166, 121)(156, 112)
\bezier{500}(-53, 130)(-63, 121)(-53, 112)
\bezier{500}(-53, 87)(-63, 78)(-53, 69)
\emline{-53}{87}{1}{156}{87}{2}
\emline{-53}{69}{1}{156}{69}{2}
\bezier{500}(156, 69)(166, 78)(156, 87)
\bezier{500}(-53, 44)(-63, 35)(-53, 26)
\emline{-53}{44}{1}{156}{44}{2}
\emline{-53}{26}{1}{156}{26}{2}
\bezier{500}(156, 44)(166, 35)(156, 26)
\emline{156}{112}{1}{-53}{112}{2}
\put(62,16){\makebox(0,0)[cc]{\small{$(3,x_j)$}}}
\put(106,94){\makebox(0,0)[cc]{\small{$(2,x_k)$}}}
\put(106,138){\makebox(0,0)[cc]{\small{$(1,x_k)$}}}
\put(62,138){\makebox(0,0)[cc]{\small{$(1,x_j)$}}}
\put(-15,138){\makebox(0,0)[cc]{\small{$(1,x_i)$}}}
\end{picture}
\def\emline#1#2#3#4#5#6{%
\put(#1,#2){\special{em:moveto}}%
\put(#4,#5){\special{em:lineto}}}
\def\newpic#1{}
\unitlength 0.23mm
\special{em:linewidth 0.4pt}
\linethickness{0.4pt}
\begin{picture}(150,150)(97,6)
%
\put(-17,78){\circle*{3}}
\put(-17,35){\circle*{3}}
\put(-17,121){\circle*{3}}
\put(60,121){\circle*{3}}
\put(60,78){\circle*{3}}
\put(60,35){\circle*{3}}
\put(103,121){\circle*{3}}
\put(103,78){\circle*{3}}
\put(103,35){\circle*{3}}
\put(-53,130){\circle*{0}}
\put(-53,87){\circle*{0}}
\put(-53,69){\circle*{0}}
\put(-53,44){\circle*{0}}
\put(-53,26){\circle*{0}}
\put(-53,112){\circle*{0}}
\put(156,69){\circle*{0}}
\put(156,87){\circle*{0}}
\put(156,44){\circle*{0}}
\put(156,26){\circle*{0}}
\put(156,112){\circle*{0}}
\put(156,130){\circle*{0}}
\emline{-17}{78}{1}{60}{35}{2}
\emline{-17}{78}{1}{103}{121}{2}
\emline{-17}{35}{1}{60}{121}{2}
\emline{-17}{35}{1}{103}{78}{2}
\emline{60}{121}{1}{103}{78}{2}
\emline{60}{78}{1}{103}{35}{2}
\emline{60}{78}{1}{-17}{121}{2}
\emline{60}{35}{1}{103}{121}{2}
\emline{-17}{121}{1}{103}{35}{2}
\emline{156}{130}{1}{-53}{130}{2}
\bezier{500}(156, 130)(166, 121)(156, 112)
\bezier{500}(-53, 130)(-63, 121)(-53, 112)
\bezier{500}(-53, 87)(-63, 78)(-53, 69)
\emline{-53}{87}{1}{156}{87}{2}
\emline{-53}{69}{1}{156}{69}{2}
\bezier{500}(156, 69)(166, 78)(156, 87)
\bezier{500}(-53, 44)(-63, 35)(-53, 26)
\emline{-53}{44}{1}{156}{44}{2}
\emline{-53}{26}{1}{156}{26}{2}
\bezier{500}(156, 44)(166, 35)(156, 26)
\emline{156}{112}{1}{-53}{112}{2}
\put(-15,94){\makebox(0,0)[cc]{\small{$(2,x_i)$}}}
\put(-15,138){\makebox(0,0)[cc]{\small{$(1,x_i)$}}}
\put(106,138){\makebox(0,0)[cc]{\small{$(1,x_k)$}}}
\put(62,138){\makebox(0,0)[cc]{\small{$(1,x_j)$}}}
\put(62,16){\makebox(0,0)[cc]{\small{$(3,x_j)$}}}
\end{picture}
\vspace*{-1mm}
\caption{Blocks of $\Omega_{2,\beta}$ and
$\Omega_{3,\beta}$}\label{type4}
\end{figure}
We note that there is only one parallel class in
$\Omega_{0,\beta}$, but there are $\frac{v-1}{2}$  parallel
classes in each of other types, so we have $\frac{3v-1}{2}$
parallel classes and each class has $v$ blocks.

Clearly, ${\cal K}$  is a ${\rm KTS}(3v)$. The number of colors
needed in a silver coloring of $1$-${\rm BIG}({\cal K})$ is equal
to $ \frac{9v-7}{2}$. We color $0$ the vertices corresponding to
the blocks in $\Omega_{0,\beta}$ class. The label of each block in
other classes, which is shown as its index, is the color of its
corresponding vertex in $1$-${\rm BIG}({\cal K})$: $(m,\beta)$,
$1 \le m \le 9$, $1\le \beta \le \frac{v-1}{2}$. It is easy to
check that this is a proper coloring and all vertices in
$\Omega_{0,\beta}$ class, i.e. the $\alpha$-set, are rainbow.
~\end{proof}


Next theorem is a generalization of the construction introduced in
Lemma~\ref{9midvKTS}.
\begin{theorem}
\label{RBIBDfork} Assume there exist an affine plane ${\cal A}=
{\rm AG }(2,k)$, and  a resolvable  balanced incomplete block
design ${\cal D}={\rm RBIBD}(v,k,1)$. Then there exists a
 ${\cal D^{*}}= {\rm RBIBD}(kv,k,1)$ where
$1$-${\rm BIG}({\cal D^*})$ is  silver.
\end{theorem}
\begin{proof}
Let $ V({\cal A})=\{ (i,j) \ | \ 1 \leq i,j \leq k \}$
 and denote its  parallel classes by
$ \Theta_0,  \Theta_1,\ldots,  \Theta_{k}$. Specifically  we let
\begin{displaymath}
\Theta_0 = \Big\{\{(1,j), (2,j), \ldots, (k,j) \}| \  j =1,2,
\ldots, k \Big\}.
\end{displaymath}
Also we let ${V(\cal D)}=\{x_1,x_2,\ldots, x_v\}$  with parallel
classes $\pi_1, \pi_2, \dots, \pi_{\frac{v-1}{k-1}}$.

For each block $b=\{x_{s_{1}},x_{s_{2}}, \ldots, x_{s_{k}}\}$ of
${\cal D}$ we consider  an ordering on $b$ such that
$$x_{s_{i}} \prec x_{s_{j}} \ \ \Longleftrightarrow  \ \ s_i < s_j,$$ and
define a function:
$$\Psi_b: V({\cal A}) \rightarrow \{1,2, \ldots, k \} \times
\{x_{s_{1}},x_{s_{2}}, \ldots,x_{s_{k}}\}$$
$$\Psi_b(i,j) =  (i,x_{s_{j}}).$$
We extend $\Psi_b$ for each block $a$ of ${\cal A}$ as \
$\Psi_b(a)=\{ \Psi_b(i,j)|  \ (i,j) \in a \}.$

Now we construct a design ${\cal D^{*}}  = (V^*, {\cal B^{*}})$,
as in the following:\\
\centerline{
$
\begin{array}{lll}
V^* = \{1,2,\ldots,k\} \times V({\cal D}).\\
{\cal B^{*}} = \{\Psi_b(a)| \   b \  {\rm and} \ a \ {\rm are \
blocks \ of \ {\cal D} \ and \ {\cal A},\ respectively}  \}.
\end{array}$}

\noindent
See Figure~\ref{mortezaeefar}.

\begin{figure}[ht]
  \setlength{\unitlength}{1bp}%
  \begin{picture}(388.30, 180.11)(0,0)
  \put(0,0){\includegraphics{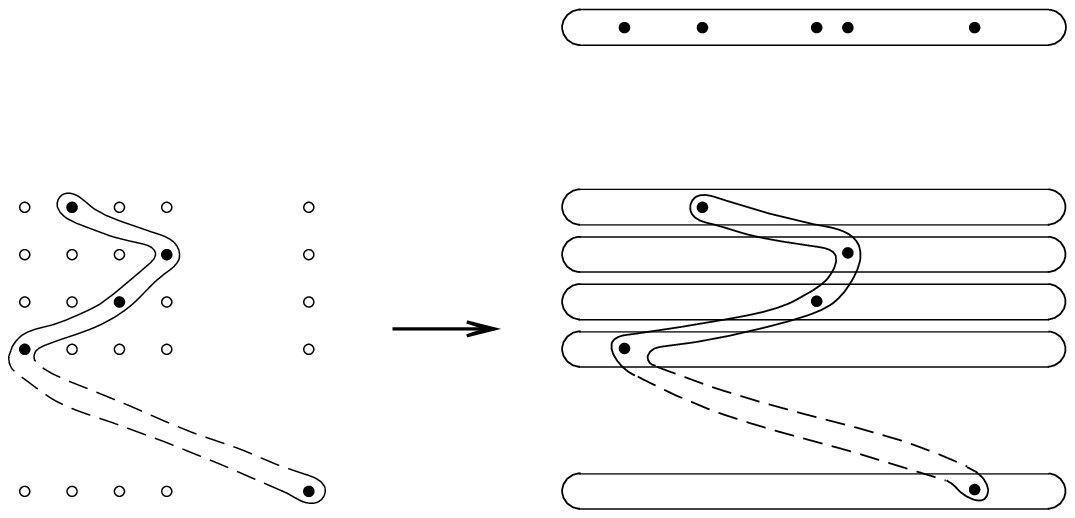}}
  \put(70.61,111.36){\fontsize{14.23}{17.07}\selectfont $\cdots$}
  \put(70.61,97.73){\fontsize{14.23}{17.07}\selectfont $\cdots$}
  \put(70.61,84.09){\fontsize{14.23}{17.07}\selectfont $\cdots$}
  \put(70.61,70.46){\fontsize{14.23}{17.07}\selectfont $\cdots$}
  \put(70.61,29.55){\fontsize{14.23}{17.07}\selectfont $\cdots$}
  \put(264.36,163.33){\fontsize{14.23}{17.07}\selectfont $\cdots$}
  \put(272.81,50.69){\fontsize{14.23}{17.07}\selectfont $\vdots$}
  \put(15.73,50.01){\fontsize{14.23}{17.07}\selectfont $\vdots$}
  \put(29.37,50.01){\fontsize{14.23}{17.07}\selectfont $\vdots$}
  \put(43.00,50.01){\fontsize{14.23}{17.07}\selectfont $\vdots$}
  \put(56.64,50.01){\fontsize{14.23}{17.07}\selectfont $\vdots$}
  \put(97.54,50.01){\fontsize{14.23}{17.07}\selectfont $\vdots$}
  \put(131.98,86.70){\fontsize{11.38}{13.66}\selectfont $\Psi_b$}
  \put(5.67,122.02){\fontsize{8.54}{10.24}\selectfont $(1,1)$}
  \put(97.44,122.02){\fontsize{8.54}{10.24}\selectfont $(1,k)$}
  \put(5.67,19.42){\fontsize{8.54}{10.24}\selectfont $(k,1)$}
  \put(97.44,18.98){\fontsize{8.54}{10.24}\selectfont $(k,k)$}
  \put(208.24,175.76){\fontsize{8.54}{10.24}\selectfont $x_{s_2}$}
  \put(184.24,175.76){\fontsize{8.54}{10.24}\selectfont $x_{s_1}$}
  \put(285.24,175.76){\fontsize{8.54}{10.24}\selectfont $x_{s_k}$}
  \put(202.24,125.76){\fontsize{8.54}{10.24}\selectfont $(1,x_{s_2})$}
  \put(281.59,16.10){\fontsize{8.54}{10.24}\selectfont $(k,x_{s_k})$}
  \put(53.03,8.12){\fontsize{11.38}{13.66}\selectfont ${\cal A}$}
  \put(242.13,8.12){\fontsize{11.38}{13.66}\selectfont ${\cal D^{*}}$}
  \put(242.13,144.33){\fontsize{11.38}{13.66}\selectfont ${\cal D}$}
  \end{picture}%
\caption{Blocks of ${\cal D^*}$ are constructed by using blocks
of ${\cal A}$} \label{mortezaeefar}
\end{figure}

\noindent
${\cal  D^{*}}$ is an {\rm RBIBD} with the following parallel
classes:

\centerline{$\Omega_{\alpha,\beta} = \{ \Psi_b(a)| \  a \in
\Theta_{\alpha}, b \in \pi_{\beta}\},$ \
for each $ 0 \leq \alpha \leq k$ and  $1 \leq \beta \leq
\frac{v-1}{k-1}.$}
\noindent
Note that:
$$\Omega_{0,1} = \Omega_{0,2}= \cdots =
\Omega_{0,\frac{v-1}{k-1}}= \Big\{ \{(1,x_{s}), (2,x_{s}),
\ldots, (k,x_{s}) \}| \  s =1,2, \ldots, v \Big\}.$$
We show that $1$-${\rm BIG}({\cal D^{*}})$ is silver with respect
to the $\alpha$-set

\centerline{$\begin{array}{lll} I^*&=& \{\Psi_b(a)| \  a \in
\Theta_0 \ \ {\rm and}   \ b \ \textrm{ is \ a
 \ block \ of } \ {\cal D} \} \\
 & =&
\Big\{ \{(1,x_{s}), (2,x_{s}), \ldots, (k,x_{s}) \}| \  s =1,2,
\ldots, v \Big\},
\end{array}$}
\noindent
by  the following coloring:
$$c:{\cal B^{*}} \longrightarrow \{0 \} \cup
\{(a,\beta)|  \  a \ {\rm is \ a \ block \  of \  {\cal A}\setminus  \Theta_0 \  and}
 \ 1 \leq \beta \leq \frac{v-1}{k-1} \}$$
\begin{displaymath}
 \Psi_b(a) \longmapsto
\left\{ \begin{array}{lll}
0             &  \hspace{.0cm} \textrm{if $a\in  \Theta_0 $},\\  \\
(a,\beta)     &  \hspace{.0cm} \textrm{if $a \notin  \Theta_0,$
and $b \in \pi_{\beta}$}.
\end{array} \right.
\end{displaymath}

We show that $c$ is a proper coloring and  any vertex $b^* \in
I^*$ is rainbow. Note that all the vertices of $I^*$ have color
$0$. Let $\Psi_{b_{1}}(a_{1})$ and $\Psi_{b_{2}}(a_{2})$ be two
blocks of ${\cal D^*}$
 with  the same color $(a,\beta)$.
Then we have $b_{1},b_{2}\in \pi_{\beta}$.
 Therefore  $b_{1}$ and $b_{2}$ are disjoint blocks of ${\cal D}$, so
$\Psi_{b_{1}}(a_{1})$ and $\Psi_{b_{2}}(a_{2})$ are disjoint.
Thus $c$ is proper.

To show silverness,  for a fixed $s$ let $b_s^*=\{(1,x_{s}),
(2,x_{s}), \ldots, (k,x_{s}) \}$ be a block of $I^*$. By
definition, for any given nonzero color like $(a,\beta)$ we have
$a \notin \Theta_0$, and  there exists a unique block $b$ of $
\pi_{\beta}$ which contains $x_{s}$ and the color of $\Psi_{b}(a)$
is  $(a,\beta)$. Since in  ${\cal A}$, the block $a$ intersects
each block of $\Theta_0$, thus by  definition of ${\cal B^*}$,
$\Psi_{b}(a)$ intersects $b_s^*$ in ${\cal D^*}$, so the color
$(a,\beta)$ appears in the neighborhood of $b_s^*$.
\end{proof}

In the next theorem for any ${\cal{D}} =S(2,k,v)$, we show a lower
bound for $\alpha(G_1)$,
 in order $G_1=
1$-${\rm BIG(\cal{D}})$ to be silver.
\begin{theorem}
\label{lowerbound}
Let $\cal{D}$ be an $S(2,k,v)$, and   $G_1= 1$-${\rm BIG(\cal{D}})$.
If $\alpha(G_1) > k \lfloor
\frac{v(v-1)}{k^2 v - k^3 + k^2- k}\rfloor$,  then $G_1$ is not
silver.
\end{theorem}
\begin{proof}
$G_1$ is a $\frac{k(v-k)}{(k-1)}$--regular graph with
$\frac{v(v-1)}{k(k-1)}$ vertices. Let $I$ be an $\alpha$-set, and assume that
 $G_1$ has a silver coloring
with respect to $I$ with $C$ as  the set of colors,
$|C|=\frac{k(v-k)}{k-1}+1$. A color like $\iota$ exists that is
used in the coloring of at most $\lfloor
\frac{|V(G_1)|}{|C|}\rfloor=\lfloor \frac{v(v-1)}{k^2 v - k^3 +
k^2- k}\rfloor$ vertices of $G_1$. For a set $X \subseteq V(G_1)$
we denote the set of vertices with color $\iota$ in $X$ by
$X(\iota)$. By counting the number of appearances of color
$\iota$ in $I$ and in the neighborhood of $I$ we obtain,

$\alpha(G_1) = |I(\iota)| + |X_1(\iota)| + 2|X_2(\iota)| +
\cdots + k
|X_k(\iota)|$ \\
\hspace*{13mm} $\leq k(|I(\iota)| + |X_1(\iota)| + |X_2(\iota)| +
\cdots +
|X_k(\iota)|)$ \\
\hspace*{13mm} $\leq k \lfloor {\frac{v(v-1)}{k^2 v - k^3 + k^2- k}}
\rfloor$ \\
\hspace*{13mm} $< \alpha(G_1).$
\\
A  contradiction.~\end{proof}
\begin{example}
\label{10sts} It is easy to check that for any of two ${\rm
STS}(13)$s, $\alpha(G_1)
=4$. 
For $80$ nonisomorphic
${\rm STS}(15)$s,  we have $\alpha(G_1) =4$ or
$5$~{\rm(see~\cite{MR2246267}, page 32)}. Also there are $18$
nonisomorphic $S(2,4,25)$~{\rm(see~\cite{MR2246267}, page 34)},
by a computer search they have
 $\alpha(G_1) =5 \ {\rm or} \ 6$.
So by Theorem~\ref{lowerbound}  none of them  has a silver $G_1$.
\end{example}
\begin{remark}
\label{remk1} Let $G_1$ be the $1$-block intersection graph of an
$S(2,k,v)$ with a parallel class. Then $\alpha(G_1) =
\frac{v}{k}$, and all the  elements of $V$ appear in the blocks
corresponding to each $\alpha$-set. Let $I$ be an $\alpha$-set for
$G_1$, therefore any vertex   of $V(G_1)\setminus I$ is adjacent
to $k$ vertices of $I$. Thus  $|X_1|=|X_2|=\cdots = |X_{k-1}|=0,
\  |X_k|= \frac{v(v-k)}{k(k-1)}.$

If an $S(2,k,v)$ has a near parallel class, then $\alpha(G_1) =
\frac{v-1}{k}$, and each $\alpha$-set contains all the  elements of $V$
 except one. Hence in this case any vertex  of
$V(G_1)\setminus I$ is adjacent to either $(k-1)$ or $k$ vertices
of $I$, and $|X_1|=|X_2|= \cdots = |X_{k-2}|=0, \  |X_{k-1}|=
\frac{v-1}{k-1}, \ |X_k|= \frac{(v-1)(v-2k+1)}{k(k-1)}$.
\end{remark}
\begin{theorem}
\label{APC}
Let $\cal{D}$ be an $S(2,k,v)$, with a near parallel
class. Then $G_1= 1$-${\rm BIG(\cal{D}})$ is not silver.
\end{theorem}
\begin{proof}
Let $I$ be an $\alpha$-set for $G_1$. Assume that $G_1$
has a silver coloring with respect to $I$ and $C$ is the set of
colors. $G_1$ is $\frac{k(v-k)}{k-1}$--regular,
$|C|=\frac{k(v-k)}{k-1}+1$ and $|I| = \frac{v-1}{k}$.
By  Remark~\ref{remk1}, $|X_{k-1}|= \frac{v-1}{k-1}$ and
 $|X_k|= \frac{(v-1)(v-2k+1)}{k(k-1)}$.
Since $|C| > |I \cup X_{k-1}|$, a color like $\iota$ exists that
is  used only in the coloring of vertices of $X_k$. The vertices
of $I$ are rainbow, so each of the vertices of  $X_k$ that have
color $\iota$, must be adjacent to $k$ different vertices of $I$.
Thus $|I|$ is a multiple of $k$, say $|I| = mk$.

Since $|X_{k-1}|= \frac{v-1}{k-1} > |I|$,   a color like $\iota'$
exists that is  used in the coloring of vertices of $X_{k-1}$
but  is not used in $I$. The induced subgraph on $ X_{k-1}$ is a
clique,
 so $\iota'$ appears only in
one vertex of $X_{k-1}$ and it has $(k-1)$ neighbors in $I$. Thus
$|I|-k+1$ vertices of $I$, each must have a neighbor in $X_k$
with color $\iota'$. Again vertices from $X_k$ that have color
$\iota'$, each must be adjacent to $k$ different vertices of $I$.
Therefore $|I|-k+1= (m-1)k+1$ is also a multiple of $k$. This is
impossible. ~\end{proof}

\begin{example}
The $1$-block intersection graph of any  Hanani triple system
{\rm(see~\cite{MR2246267}, page 67 for the definition)} is not
silver.
\end{example}
Note that by~Theorems~\ref{forallk},~\ref{RBIBDfork},~\ref{lowerbound},
and~\ref{APC},
for any admissible $v$ the problem of existence of a silver $1$-{\rm BIG($S(2,k,v)$)}
which possesses maximum possible independent set is settled.
\section{Zero block intersection graphs}
In this section we discuss $0$-block intersection graphs of
$S(2,k,v)$.\\

\noindent
\begin{notation}
\label{remk0} Let $x$ be a given element of  $S(2,k,v)$, and
denote  by $T(x)$ the set of $\frac{v-1}{k-1}$ blocks containing
$x$.
\end{notation}
It is trivial that $T(x)$  is an  independent set for $G_0$, thus $
\alpha(G_0) \geq \frac{v-1}{k-1}$.

\begin{lemma}
\label{lemmaG0}
Let $\cal{D}$ be an $S(2,k,v)$,
and $G_0 = 0$-${\rm BIG(\cal{D}})$. If $v >  k^3-2k^2+2k$
then  any maximum independent set of $G_0$ is of
the form $T(x)$, therefore   $\alpha(G_0) = \frac{v-1}{k-1}$.
\end{lemma}
\begin{proof}
Let $I$ be an $\alpha$-set of $G_0$. Suppose   $I$ is
not of the form $T(x)$. There exists an element $x_0$  of
$\cal{D}$ which appears in at least  two blocks of $I$. Let
$I_1=\{B_1, B_2, \ldots, B_p\}=\{B | \ B \in I\cap T(x_0)\}$, and
 $I\setminus I_1 =\{B_{p+1}, B_{p+2}, \ldots, B_{p+q}\}$. Since $ \lambda =1$,
\ for $1\leq i<j \leq p, $ \ $(B_i\setminus \{x_0\})\cap
(B_j\setminus \{x_0\}) = \emptyset$. Every  two blocks in $I$
have one  intersection. So, for each block  $B \in I\setminus
I_1$ \ we have $B\cap B_i = \{a_i\}$, $i= 1,2,\dots, p $. So
$p \leq |B|=k$.

Now suppose $B_1, B_2 \in I_1$. There exist exactly $(k-1)^2$
pairs  $\{x,y\}$ where $ x \in B_1\setminus\{x_0\}$ and $ y \in
B_2\setminus\{x_0\}$, and each of these pairs appears  at most in one of
the blocks of $I\setminus I_1$. Thus $ q \leq (k-1)^2$.

So $|I| = p+q \leq k +(k-1)^2$. But since $v > k^3-2k^2+2k$, for
each $x$ we have $|T(x)|= \frac{v-1}{k-1} > k +(k-1)^2 \geq |I|$.
Hence the statement follows.~\end{proof}

\begin{theorem}
\label{G0forallk}
Let $\cal{D}$ be an $S(2,k,v)$.
For $v > k^3-2k^2+2k $, $G_0 = 0$-${\rm BIG(\cal{D}})$ is not silver.
\end{theorem}
\begin{proof}
$G_0$ is a $\frac{v^2+k^3-v(k^2+1)-k^2+k}{k(k-1)}$--regular graph
~(Remark~\ref{srg1}). Let $I$ be any $\alpha$-set for
$G_0$. By Lemma~\ref{lemmaG0}, $I = T(x)$ and $|I|= \alpha(G_0) =
\frac{v-1}{k-1}$. Since each block out of  $I$ intersects exactly
$k$  blocks of $I$, each vertex
 of $V(G_0) \setminus I$ is adjacent to $\frac{v-1}{k-1}-k =
\frac{v-1-k^2+k}{k-1}$  vertices of $I$. Then $ V(G_0)= I \cup
X_{\frac{v-1-k^2+k}{k-1}}$ and $|X_{\frac{v-1-k^2+k}{k-1}}|=
\frac{(v-1)(v-k)}{k(k-1)}.$

To the contrary, $G_0$ has a silver coloring  with respect to
$I$. Let $C$ be the set of colors, $|C|=
\frac{v^2-v-k^2v+k^3}{k(k-1)}$. Since $|C| > \frac{v-1}{k-1}$, a
color like $\iota$  exists that is not used in the coloring of
$I$. The vertices of $I$ are rainbow, and  the vertices from
$X_{\frac{v-1-k^2+k}{k-1}}$ that have color $\iota$,
 each must be adjacent to $\frac{v-1-k^2+k}{k-1}$
different vertices of $I$. Therefore $|I|$  must be divisible by
$\frac{v-1-k^2+k}{k-1}$, then $(v-k^2+k-1) \mid (v-1) $ which is
impossible, since $v > k^3-2k^2+2k$. Therefore graph $G_0$ is not
silver with respect to any $\alpha$-set.
\end{proof}
%
\subsection{$0$-{\rm BIG} for  Steiner triple systems}
Both $0$-${\rm BIG (STS}(v))$   for $v=7$ and $v=9$, by
Example~\ref{symmetric-AG},  are totally silver.
\begin{theorem}
For any admissible $v>9$, $G_0 = 0$-{\rm BIG(STS$(v)$)} is not
silver.
\end{theorem}
\begin{proof}
For  $ v>15$, it follows by Theorem~\ref{G0forallk}.\\
If $ v \leq 15$, then suppose $I$ is an $\alpha$-set of
$G_0$, and  $I$ is not of the form $T(x)$. Then it is easy to
check that, each element of ${\rm STS}(v)$ appears at most in $3$
blocks of $I$. If it has $3$ blocks containing an element $x$,
then such a  set has at most $7$ blocks, and they are contained
in $I_1$, where:
$$I_1=\{ \{x,a,b\},\{x,c,d\},\{x,e,f\},\{a,c,f\}, \{a,d,e\},\{b,c,e\},\{b,d,f\}\}
\approx {\rm STS}(7).$$
Now we discuss possible cases.
\\

\noindent
{\bf\Large $v=15$:}\\
For  $v=15$ an $\alpha$-set, $I$, may be of   the form
$T(x)$ or it may come from a subsystem ${\rm STS}(7)$, in either
case $\alpha(G_0)= 7$. From $80$ non-isomorphic  ${\rm STS}(15)$s,
$23$ of them have a subsystem ${\rm STS}(7)$
{\rm(\cite{MR2246267}}, page 32). It is straightforward to check
that in all of ${\rm STS}(15)$s for any $\alpha$-set
$I$, each block out of $I$ has intersection with exactly three
blocks of $I$. So each vertex in $V(G_0)\setminus I$ is adjacent
to exactly four vertices of $I$. In any silver coloring with $C$
as the set of colors of $G_0$, we have $|C|=17 >7= |I|$. So there
exists a color $\iota$ which  is not used in $I$. Every vertex
with the color $\iota$ has exactly $4$ neighbors in $I$,
therefore $7$ must be a multiple of $4$. So $G_0$ does not have a
silver coloring.
\\

\noindent
{\bf \Large $v=13$:}\\
For $v=13$ there are two non-isomorphic ${\rm STS}(13)$s. No ${\rm
STS}(13)$  has a subsystem of ${\rm STS}(7)$, even no ${\rm
STS}(13)$ has $6$ blocks of an ${\rm STS}(7)$. So, in $G_0$ for
both of them, the sets  of the form $T(x)$, are the only $\alpha$-sets and $\alpha(G_0)=6$.
Suppose $I$ is any $\alpha$-set. \\
First, we show that it is always possible to find three vertices
in $I$ with no common neighbor:
\begin{itemize}
\item
One of  two  ${\rm STS}(13)$s, Type $1$, has a cyclic
automorphism, and we can construct its  blocks on $\{1,2,\dots,13\}$ by the following
 base blocks:
$$\{1,2,5\}, \ \ \ \{1,3,8\} \ \quad {\rm mod} \  13.$$
If  $I= T(1)$, then $B_1= \{1,2,5\}$, $B_2=\{1,3,8\},$ and  $B_3=
\{1,10,11\}$ do not have common neighbor. Let $x \neq 1$ be a
given element of  ${\rm STS}(v)$, and $I= T(x)$. Three vertices
of $I$, $B'_1,B'_2,B'_3$ are obtained by adding $(x-1)$ to all
members of blocks $B_1,B_2,B_3$, do not have common neighbor.
\item
The other ${\rm STS}(13)$ is non-cyclic and we can construct its
blocks from Type $1$ by  replacing   four blocks of trade $T_1$
with  four blocks of  trade $T_2$ as follows:
$$T_1: \ \ \begin{array}{ccc}
1&2&5\\
1&3&8 \\
10&2&8 \\
10&3&5 \\
\end{array} \hspace{2cm}
T_2: \ \  \begin{array}{ccc}
1&2&8\\
1&3&5 \\
10&2&5 \\
10&3&8 \\
\end{array}$$

Let  $I=T(x)$ for some $x$. If  $x$ is an element of  $T_2$, i.e.
$ x \in \{1,2,3,5,8,10\}$, then there are two blocks say $B_1$ and
$B_2$ of  $T_2$ which contain $x$. There exists one element $y$,
such that $ y \in T_2$ but  $y \notin B_1\cup B_2$. We consider
$B_3$, the block containing $x$ and $y$. Then these three blocks
do not have common neighbor. If $x$ is not in $T_2$, then we
consider several cases for $I=T(x)$, and show that there exist
three vertices of $I$, which do not have common neighbor.
\end{itemize}
Now, assume for some STS$(13)$, \ $G_0 = 0$-{\rm BIG(STS$(13)$)}
is silver with respect to some $\alpha$-set
$I=T(x)=\{B_1,B_2,B_3,B_4,B_5,B_6 \}$. The color of all neighbors
of $B_i$, $i=1,\dots,6$, must be distinct. Assume
$\{B_1,B_2,B_3\}\subset I$ do not have common neighbor. Let
$N(B_i)$ be the set of neighbors of $B_i$. $G_0=SRG(26,10,3,4)$,
so
 $ |N(B_1) \cap N(B_2)| + |N(B_2) \cap N(B_3)| + |N(B_1) \cap N(B_3)| = 12.$
Thus the color of these vertices must be distinct, while we have
only $11$ colors. Therefore $G_0$ does not have a silver
coloring.~\end{proof}
\def\cprime{$'$} \def\cprime{$'$} \def\cprime{$'$}

\end{document}